\documentclass[a4paper,reqno,12pt]{amsart}
\usepackage[english]{babel}
\usepackage{amsmath}
\usepackage{amssymb}
\usepackage{amscd}
\usepackage{amsthm}
\usepackage{euscript}
\newtheorem{proposition}{Proposition}
\newtheorem{lemma}{Lemma}
\newtheorem{theorem}{Theorem}

\theoremstyle{definition}
\newtheorem{example}{Example}
\theoremstyle{remark}
\newtheorem {remark}{Remark}

\DeclareMathOperator{\spec}{Spec}

\DeclareMathOperator{\Aut}{Aut}

\DeclareMathOperator{\GL}{GL}

\def\Ker{{\rm Ker}}
\def\Im{{\rm Im}}

\def\HH{{\mathbb H}}
\def\GG{{\mathbb G}}

\def\CC{{\mathbb C}}
\def\KK{{\mathbb K}}
\def\TT{{\mathbb T}}
\def\ZZ{{\mathbb Z}}

\def\QQ{{\mathbb Q}}
\def\PP{{\mathbb P}}
\def\AA{{\mathbb K}}

\begin{document}
\date{}
\title[The automorphism group of a rigid variety]{The automorphism group of a rigid affine variety}
\author{Ivan Arzhantsev}
\thanks{The research of the first author was supported by the grant RSF-DFG 16-41-01013.}
\address{National Research University Higher School of Economics, Faculty of Computer Science, Kochnovskiy Proezd 3, Moscow, 125319 Russia}
\email{arjantsev@hse.ru}
\author{Sergey Gaifullin}
\address{Lomonosov Moscow State University, Faculty of Mechanics and Mathematics, Department of Higher Algebra, Leninskie Gory 1, Moscow, 119991 Russia; \linebreak and \linebreak
National Research University Higher School of Economics, Faculty of Computer Science, Kochnovskiy Proezd 3, Moscow, 125319 Russia}
\email{sgayf@yandex.ru}



\dedicatory{To Mikhail Zaidenberg on occasion of his 70th birthday}

\maketitle

\begin{abstract}
An irreducible algebraic variety $X$ is rigid if it admits no nontrivial action of the additive group of the ground field. We prove that the automorphism group $\Aut(X)$ of a rigid affine variety contains a unique maximal torus $\TT$. If the grading on the algebra of regular functions $\KK[X]$ defined by the action of $\TT$ is pointed, the group $\Aut(X)$ is a finite extension of $\TT$. As~an~application, we describe the automorphism group of a rigid trinomial affine hypersurface and find all isomorphisms between such hypersurfaces.
\end{abstract}

\section{Introduction}

Let $\KK$ be an algebraically closed field of characteristic zero and $X$ an algebraic variety over~$\KK$. Consider the group $\Aut(X)$ of all algebraic automorphisms of the variety $X$. In~general, $\Aut(X)$ is not an algebraic group. We say that a subgroup $G$ in $\Aut(X)$ is {\it algebraic} if $G$ admits a structure of an algebraic group such that the action $G\times X\to X$ is a morphism of algebraic varieties. This way we speak about subtori, semisimple subgroups and unipotent subgroups in the group $\Aut(X)$.

Assume that $X$ is an affine variety. An intriguing problem is to describe explicitly the automorphism group $\Aut(X)$. It is easy to see that the automorphism group of the affine line is a 2-dimensional linear algebraic group. The automorphism group of the affine plane is infinite-dimensional, and its structure as a free amalgamated product is given  by the classical Jung--van der Kulk theorem. Starting from dimension 3, the structure of the automorphism group of the affine space is mysterious, and it attracts a lot of attention.

Let $\GG_a:=(\KK,+)$ be the additive group of the ground field. Let us recall that a derivation $D$ of an algebra $A$ is called \emph{locally nilpotent}, if for any $a\in A$ there exists a positive integer $m$ such that $D^m(a)=0$.
It is well known that regular $\GG_a$-actions on an irreducible affine variety $X$ are in natural bijection with locally nilpotent derivations of the algebra of regular functions $\KK[X]$, see~\cite[Section~1.5]{F}. If $X$ admits a non-trivial $\GG_a$-action and
$\dim(X)\ge 2$, then the group $\Aut(X)$ is infinite dimensional.  Indeed, if $D$ is a nonzero locally nilpotent derivation on the algebra $\KK[X]$, its kernel $\Ker(D)$ has transcendence degree $n-1$ \cite[Principle~11]{F} and, taking exponentials of locally nilpotent derivations of the form $fD$, $f\in\Ker(D)$, one obtains commutative unipotent subgroups in $\Aut(X)$ of arbitrary dimension.

An irreducible affine variety is called \emph{rigid}, if it admits no non-trivial $\GG_a$-action. Equivalently, the algebra $\KK[X]$ admits no non-zero locally nilpotent derivation. The class of rigid varieties is studied actively during last decades; see e.g. \cite{KZ0,CML,BML,CM,FMJ,KPZ1,Ar} for both algebraic and geometric results and examples.

Let us recall that the \emph{Makar-Limanov invariant} $\text{ML}(A)$ of an algebra $A$ is the intersection of kernels of all locally nilpotent derivations on $A$. Clearly, an affine variety $X$ is rigid if and only if $\text{ML}(\KK[X])=\KK[X]$. A systematic study of the Makar-Limanov invariant and rigidity may be found in~\cite[Chapter~9]{F}.

A wide class of rigid varieties form so-called toral varieties \cite{Po}.
A variety is said to be \emph{toral} if it is isomorphic to a closed subvariety of an algebraic torus. Equivalently, an affine variety is toral if the algebra $\KK[X]$ is generated by invertable elements \cite[Lemma~1.14]{Po}. Since every locally nilpotent derivation annihilates all invertible elements \cite[Corollary~1.20]{F}, every toral variety is rigid.

In~\cite{KPZ1}, a geometric criterion for rigidity of an affine cone over a projective variety is given. Rigidity of some affine cones over del Pezzo surfaces of degree at most $3$ is proved in \cite{KPZ2, CPW}.

In this paper we study the automorphism group $\Aut(X)$ of a rigid affine variety $X$. It is suitable for our purposes to use the following convention: a group $G$ is an extension of
a subgroup $H$ by a group $F$ if $G$ contains $H$ as a normal subgroup and the factor group $G/H$ is isomorphic to $F$. An algebraic torus $T$ is an algebraic group isomorphic to the direct product of several copies of the multiplicative group $\KK^{\times}$ of the ground field. A quasitorus, or a diagonalizable group, is an algebraic group isomorphic to the direct product of a torus and a finite abelian group.

In Section~\ref{s2} we show that the automorphism group $\Aut(X)$ of a rigid affine variety contains a unique maximal torus $\TT$ (Theorem~\ref{tori}). Moreover, the group $\Aut(X)$ is an extension of the centralizer of $\TT$ in the group $\Aut(X)$ by a subgroup of the discrete group $\GL_r(\ZZ)$.

The action $\TT\times X\to X$ induces a canonical grading on the algebra $\KK[X]$. If this grading is pointed, the group $\Aut(X)$ is a finite extension of $\TT$ (Proposition~\ref{finext}). These results allow to find the automorphism group of an affine cone over a projective variety with a finite automorphism group provided the cone is rigid (Proposition~\ref{cone}). In particular, they confirm Conjecture~1.3 in~\cite{KPZ1} on the automorphism group of an affine cone over a del Pezzo surface.

We apply our technique to describe the automorphism group of a rigid trinomial affine hypersurface. Examples of rigid trinomial hypersurfaces can be found in~\cite{KZ0, CM, FMJ}. In~\cite[Theorem~1.1]{Ar}, a criterion for rigidity of a factorial trinomial hypersurface is obtained. It turns out that many trinomial hypersurfaces are rigid.

Our interest in trinomials is motivated by toric geometry and Cox rings. Let a torus $T$ act effectively on an irreducible variety $X$. The \emph{complexity} of such an action is the codimension of a general $T$-orbit in $X$. Actions of complexity zero are actions with an open $T$-orbit. A~normal variety admitting a torus action with an open orbit is called a toric variety.

An important invariant of a variety is its total coordinate ring, or the Cox ring; see \cite{ADHL} for details. By \cite{Cox}, the Cox ring of a toric variety is a polynomial algebra. In turn, it is shown in \cite{HS, HHS, HH, HW} that the Cox ring of a rational variety with a torus action of complexity one is a factor ring of a polynomial ring by an ideal generated by trinomials. The study of homogeneous locally nilpotent derivations on such rings lead to a description of the automorphism group of a complete rational variety with a complexity one torus action \cite{AHHL}.

Theorem~\ref{main} claims that the automorphism group of a rigid affine trinomial hypersurface can be obtained in two steps. At the first step we extend the maximal torus $\TT$ by a finite abelian group and obtain a quasitorus $\HH$ of all automorphisms acting on the variables by scalar multiplication. At the second step we extend $\HH$ by a (finite) group $P(f)$ of all permutations of the variables preserving the trinomial $f$. The first step is not needed exactly when the trinomial hypersurface is factorial. In this case the group $\Aut(X)$ coincides with the semi-direct product $P(f)\rightthreetimes \TT$ (Theorem~\ref{fact}).

In the last section, we find all isomorphisms between rigid trinomial hypersurfaces. Namely, two such hypersurfaces are isomorphic if and only if the corresponding trinomials are obtained from each other by a permutation of variables.

\section{Subtori in the automorphism group} \label{s2}

In this section we show that rigidity imposes strong restrictions on subtori in the group $\Aut(X)$. The proof of following theorem uses an idea and techniques from~\cite[Section~3]{FZ2}.

\begin{theorem} \label{tori}
Let $X$ be a rigid affine variety. There is a subtorus $\TT$ in $\Aut(X)$ that contains any other subtorus of $\Aut(X)$. In particular, $\TT$ is a normal subgroup of $\Aut(X)$.
\end{theorem}

\begin{proof}
If $\TT$ is a subtorus in $\Aut(X)$, then $\dim\TT\le\dim X$. Let $\TT$ be a subtorus of maximal dimension. Assume that there is a subtorus $\TT'$ in $\Aut(X)$ which is not contained in~$\TT$. Then there are one-parameter subgroups $\TT_1\subseteq\TT$ and $\TT_1'\subseteq\TT'$ which do not commute. These subgroups define two $\ZZ$-gradings on the algebra $A=\KK[X]$:
$$
A=\bigoplus_{u\in\ZZ} A_u \quad \text{and} \quad A=\bigoplus_{u\in\ZZ} A'_u
$$
and two semisimple derivations $D,D'\colon A\to A$ with
$$
D(a)=ua, \quad a\in A_u \quad \text{and} \quad D'(a')=ua', \quad a'\in A'_u.
$$
Since the algebra $A$ is finitely generated, the derivation $D'$ can be decomposed into a finite sum of derivations homogeneous with respect to the first grading. By assumption, $\TT_1$ and $\TT_1'$ do not commute, hence there is a nonzero homogeneous component $D'_w$ of $D'$ with $w\in\ZZ\setminus\{0\}$. Let us take $D'_w$, where $w$ has maximal absolute value.

For every $a\in A$ there is a finite dimensional $D'$-invariant subspace $U\subseteq A$ that contains~$a$. If $a$ is homogeneous for the first grading and $(D'_w)^m(a)\ne 0$, then $(D'_w)^m(a)$ coincides with the highest (or the lowest) component of $(D')^m(a)$. But the subspace $U$ can not project non-trivially to infinitely many components $A_u$. This shows that $(D'_w)^m(a)=0$ for $m\gg 0$ and thus $D'_w$ is a nonzero locally nilpotent derivation, a contradiction.
\end{proof}

\begin{remark}
Using different techniques, it is proved in \cite[Proposition~4.4]{KPZ} that if the neutral component of the automorphism group of a rigid affine surface is an affine algebraic group, then this component is a torus. Let us notice that surfaces of class ($\text{ML}_2$) in terminology of \cite{KPZ} are precisely rigid surfaces in our terms.
\end{remark}

Let $C$ be the centralizer of the maximal subtorus $\TT$ in $\Aut(X)$. Clearly, $C$ is a normal subgroup of $\Aut(X)$. Denote by $F$ the factor group $\Aut(X)/C$. Then $F$ is naturally identified with a subgroup of the group $\Aut_{\text{gr}}(\TT)$ of  group automorphisms of $\TT$. In turn, the group $\Aut_{\text{gr}}(\TT)$ may be identified with the group $\GL(M)$ of automorphisms of the character lattice $M$ of the torus $\TT$. If $r=\dim\TT$, then $\GL(M)$ is isomorphic to the group $\GL_r(\ZZ)$. In particular, $F$ is a discrete group.

\begin{example}
Let $X$ be an affine variety isomorphic to an algebraic torus $\TT$ of dimension $n$. Since the algebra $\KK[X]$ is generated by invertible functions $T_1,T_1^{-1},\ldots,T_n,T_n^{-1}$, any locally nilpotent derivation on $\KK[X]$ is zero \cite[Corollary~1.20]{F} and $X$ is rigid.

Clearly, the maximal torus in $\Aut(X)$ can be identified with $\TT$. Since $\TT$ acts on $T_1,\ldots,T_n$ by linearly independent characters, the centralizer $C$ of $\TT$ coincides with $\TT$. Finally, the group $F$ equals $\GL_n(\ZZ)$. Moreover, we have the semi-direct product structure
$$
\Aut(X)\cong \GL_n(\ZZ) \rightthreetimes \TT.
$$
\end{example}

\begin{remark} \label{torrid}
It is well known that the only rigid affine toric variety is the torus itself, see e.g. \cite[Section~2]{AKZ}.
\end{remark}

\section{The canonical grading}

Let $X$ be a rigid affine variety and $\TT$ the maximal torus in $\Aut(X)$. The~action $\TT\times X\to X$ defines a grading on $\KK[X]$ by the lattice of characters $M$:
$$
\KK[X]=\bigoplus_{u\in M} \KK[X]_u, \quad \KK[X]_u=\{f\in\KK[X] \, | \, t\cdot f= \chi^u(t)f \ \
\text{for all} \ t\in\TT \},
$$
where $\chi^u$ is the character of $\TT$ corresponding to a point $u\in M$. We call this grading the \emph{canonical grading} on the algebra $\KK[X]$. Since the subgroup $\TT$ is normal in $\Aut(X)$, any automorphism of the algebra $\KK[X]$ sends a homogeneous element with respect to this grading to a homogeneous one.

\smallskip

We proceed with the following definitions. Let $K$ be a finitely generated abelian group and consider a finitely generated $K$-graded integral $\KK$-algebra with unit
$$
A = \bigoplus_{w\in K} A_w.
$$
The \emph{weight monoid} is the submonoid $S(A)\subseteq K$ consisting of the elements $w\in K$ with $A_w\ne 0$. The \emph{weight cone} is the convex cone $\omega(A)\subseteq K_{\QQ}$
in the rational vector space $K_{\QQ}:=K\otimes_{\ZZ}\QQ$ generated by the weight monoid $S(A)$.
We say that the $K$-grading on $A$ is \emph{pointed} if the weight cone $\omega(A)$ contains no line and $A_0=\KK$.

\begin{remark}
One can check that if $A$ admits a pointed grading by a lattice, then the group of invertible elements $A^{\times}$ coincides with invertible constants $\KK^{\times}$.
\end{remark}

For the canonical grading we have $A=\KK[X]$ and $K=M$. Let us denote the weight monoid (resp. the weight cone) by $S(X)$ (resp. by $\omega(X)$) in the case. The canonical grading is \emph{effective} in a sense that the weight monoid $S(X)$ generates the group~$K$. This reflects the fact that the action $\TT\times X\to X$ is effective.

\begin{proposition} \label{finext}
Let $X$ be a rigid affine variety. Assume that the canonical grading on $\KK[X]$ is pointed. Then the automorphism group $\Aut(X)$ is a finite extension of the~torus~$\TT$.
\end{proposition}

\begin{proof}
By~\cite[Proposition~2.2]{AHHL}, the group $\Aut(X)$ is a linear algebraic group over $\KK$ and $\KK[X]$ is a rational module with respect to the induced action $\Aut(X)\times \KK[X]\to\KK[X]$.

By rigidity, $\Aut(X)$ contains no $\GG_a$-subgroup. This implies that both the unipotent radical and the semisimple part of $\Aut(X)$ are trivial. Thus the neutral component of $\Aut(X)$ coincides with $\TT$, and $\Aut(X)$ is a finite extension of $\TT$.
\end{proof}

\section{Automorphisms of rigid cones}

In this section we give first applications of the results obtained above.

\begin{proposition} \label{cone}
Let $Z$ be a closed subvariety in $\PP^m$  with a finite automorphism group and $X$ in $\AA^{m+1}$ be the affine cone over $Z$. Assume that $X$ is rigid. Then the maximal torus in $\Aut(X)$ is the one-dimensional torus $\TT$ acting on $X$ by scalar multiplication. Moreover, the group $\Aut(X)$ is a finite extension of $\TT$ and $\TT$ is central in $\Aut(X)$.
\end{proposition}

\begin{proof}
By Theorem~\ref{tori}, the maximal torus $\TT$ in $\Aut(X)$ contains $\KK^{\times}$. Thus the action of $\TT$ on $X$ descends to the action of $\TT/\KK^{\times}$ on the projectivization $Z$.
Since the group $\Aut(Z)$ is finite, we have $\TT=\KK^{\times}$. It~follows from Theorem~\ref{tori} that $\KK^{\times}$ is normal in $\Aut(X)$. Since the $\KK^{\times}$-action on $X$ induces a $\ZZ_{\ge 0}$-grading on $\KK[X]$, the subgroup $\KK^{\times}$ is central in $\Aut(X)$ and $\Aut(X)/\KK^{\times}\subseteq\Aut(Z)$. The assertion follows.
\end{proof}

\begin{remark}
Proposition~\ref{cone} confirms \cite[Conjecture~1.3]{KPZ1} on the automorphism group of an affine cone over a del Pezzo surface.
\end{remark}

\begin{example} \label{3-4}
Consider the affine cone $X$ in $\AA^4$ given by the equation
$$
x_1^3+x_2^3+x_3^3+x_4^3=0.
$$
It is proved in \cite{CPW} that $X$ is rigid. Let $\KK^{\times}$ be the one-dimensional torus acting on $X$ by scalar multiplication and $Z$ be the projectivization of $X$ in $\PP^3$.
By~\cite[Theorem~9.5.8]{Do}, the group $\Aut(Z)$ is isomorphic to $S_4\rightthreetimes (\ZZ/3\ZZ)^3$. This fact and Proposition~\ref{cone} imply that
$$
\Aut(X)\cong S_4\rightthreetimes ((\ZZ/3\ZZ)^3\times \KK^{\times}).
$$
Thus the group $\Aut(X)$ permutes the coordinates $x_1,x_2,x_3,x_4$ and multiplies them by suitable scalars.
\end{example}

\begin{example} \label{d-3}
Let $X$ be the affine cone in $\AA^3$ given by the equation
$$
x_1^d+x_2^d+x_3^d=0.
$$
By~\cite[Lemma~4]{KZ0}, $X$ is rigid if and only if $d\ge 3$. In this case we have
$$
\Aut(X)\cong S_3\rightthreetimes ((\ZZ/d\ZZ)^2\times \KK^{\times}),
$$
and the group $\Aut(X)$ permutes the coordinates $x_1,x_2,x_3$ and multiplies them by suitable scalars. Indeed, the case $d=3$ follows from Example~\ref{3-4}. Let $d\ge 4$. The automorphism group $\Aut(Z)$ of the projectivization $Z$ of $X$ in $\PP^2$ is isomorphic to $S_3\rightthreetimes (\ZZ/d\ZZ)^2$ (see~\cite{Tz}) and Proposition~\ref{cone} implies the claim.
\end{example}

\section{The automorphism group of a rigid trinomial hypersurface} \label{s3}

Let us fix positive integers $n_0,n_1,n_2$ and let $n=n_0+n_1+n_2$. For each $i=0,1,2$, we fix a tuple $l_i\in\ZZ_{>0}^{n_i}$ and define a monomial
$$
T_i^{l_i}:=T_{i1}^{l_{i1}}\ldots T_{in_i}^{l_{in_i}}\in\KK[T_{ij}; \ i=0,1,2, \, j=1,\ldots,n_i].
$$
By a \emph{trinomial} we mean a polynomial of the form $f=T_0^{l_0}+T_1^{l_1}+T_2^{l_2}$. A \emph{trinomial hypersurface} $X$ is the zero set $\{f=0\}$ in the affine space $\KK^n$. One can check that $X$ is an irreducible normal affine variety of dimension $n-1$.

Consider an integral $2\times n$ matrix
\begin{eqnarray*}
L
& = &
\left[
\begin{array}{cccc}
-l_0 & l_1 & 0
\\
-l_0 & 0 & l_2
\end{array}
\right].
\end{eqnarray*}

Let $L^t$ denote the transpose of $L$. Consider the grading on the polynomial ring $\KK[T_{ij}]$ by the factor group $K:=\ZZ^n/\Im(L^t)$ defined as follows. Let $Q\colon\ZZ^n \to K$ be the projection. We let
$$
\deg(T_{ij}):=w_{ij}:=Q(e_{ij}),
$$
where $e_{ij}\in\ZZ^n$, $0\le i\le 2$, $1\le j\le n_i$, are the standard basis vectors. Then
all terms of the trinomial $f=T_0^{l_0}+T_1^{l_1}+T_2^{l_2}$ are homogeneous of degree
$$
\mu:=l_{01}w_{01}+\ldots+l_{0n_0}w_{0n_0}=l_{11}w_{11}+\ldots+l_{1n_1}w_{1n_1}=
l_{21}w_{21}+\ldots+l_{2n_2}w_{2n_2}
$$
and we obtain a $K$-grading on the algebra $\KK[X]=\KK[T_{ij}]/(f)$. We use the same notation for elements of $\KK[T_{ij}]$ and their projections to $\KK[X]$.

\smallskip

Let $\HH:=\spec(\KK[K])$ be the quasitorus with the character group $K$. The effective $K$-grading on $\KK[X]$ defines an effective action $\HH\times X\to X$. Let $\TT$ be the neutral component of the group $\HH$. Then $\TT$ is a torus and $\HH=\TT$ if and only if the group $K$ is free.

\begin{theorem} \cite[Theorems~1.1 and 1.3]{HH} \label{pcg}
The $K$-grading on the algebra $\KK[X]$ is effective, pointed, and the action $\TT\times X \to X$ has complexity one.
\end{theorem}

\begin{proposition} \cite[Corollary~3.6]{AHHL}
For every generator $T_{ij}$ the homogeneous component of $\KK[X]$ of degree $w_{ij}$ is of dimension one.
\end{proposition}

Further we assume that $n_il_{i1}>1$ for $i=0,1,2$ or, equivalently, $f$ does not contain a linear term; otherwise the hypersurface $X$ is isomorphic to the affine space $\KK^{n-1}$.

Let us introduce some numbers associated with a trinomial.  Let
$$
d_i:=\text{gcd}(l_{i1},\ldots,l_{in_i}), \quad d_{ij}:=\text{gcd}(d_i,d_j), \quad d:= \text{gcd}(d_0,d_1,d_2) \quad \text{and} \quad d':=dd_{01}d_{02}d_{12}.
$$

\begin{lemma} \label{quotient}
The group $\HH/\TT$ is isomorphic to $(\ZZ/d\ZZ)\times(\ZZ/d'\ZZ)$.
\end{lemma}

\begin{proof}
Using the results of \cite[Chapter~II, Section~15]{Neu}, one can easily show that the Smith normal form of the integer matrix $L$ is $\text{diag}(d,dd_{01}d_{02}d_{12})$. This implies the assertion.
\end{proof}

Let $\tau$ be a permutation of the variables $T_{ij}$. Then $\tau$ acts naturally on the polynomial algebra $\KK[T_{ij}]$. We say that $\tau$ is a permutation of a trinomial $f=T_0^{l_0}+T_1^{l_1}+T_2^{l_2}$ if $\tau(f)=f$. All such permutations form a subgroup $P(f)$ in the symmetric group $S_n$. We call this subgroup the \emph{permutation group} of the trinomial $f$.

\begin{example} \label{per}
Consider the trinomial
$$
f=T_{01}^2T_{02}^2+T_{11}^2T_{12}^2+T_{21}^3.
$$
Then the permutation group $P(f)$ is a subgroup of order $8$ in $S_5$ isomorphic to the
dihedral group $D_4$.
\end{example}

We come to a description of the automorphism group of a rigid trinomial hypersurface.

\begin{theorem} \label{main}
Let $X$ be a rigid hypersurface in $\AA^n$ given by a trinomial $f=T_0^{l_0}+T_1^{l_1}+T_2^{l_2}$, $\HH$ be the quasitorus defined above and $P(f)$ be the permutation group of the trinomial $f$. Then
$$
\Aut(X)=P(f)\rightthreetimes\HH.
$$
\end{theorem}

The proof of Theorem~\ref{main} is given in the next section.

\begin{example}
The hypersurface $X$ given by
$$
T_{01}^2T_{02}^3+T_{11}^2+T_{21}^4=0
$$
is rigid \cite[Lemma~5.15]{CM}, the torus $\TT$ is 2-dimensional and it acts on $X$ as
$$
(t_1,t_2)\cdot(x_{01},x_{02},x_{11},x_{21})=(t_1^2t_2^{-3}x_{01},t_2^2x_{02},t_1^2x_{11},t_1x_{21}).
$$
In this case the permutation group $P(f)$ is trivial and $\Aut(X)=\HH=(\ZZ/2\ZZ)\times\TT$, where
the group $\ZZ/2\ZZ$ acts as
$$
(x_{01},x_{02},x_{11},x_{21})\mapsto(x_{01},x_{02},-x_{11},x_{21}).
$$
\end{example}

\section{Proof of Theorem~\ref{main}}

Let $X$ be a rigid hypersurface in $\AA^n$ given by a trinomial $f=T_0^{l_0}+T_1^{l_1}+T_2^{l_2}$. Since $f$ contains no linear term, the variety $X$ is singular.
By Remark~\ref{torrid}, the variety $X$ is not toric, and thus the torus $\TT$ constructed in the previous section is the maximal torus in the group $\Aut(X)$.

We use notation introduced above. Let $u_{ij}$ be the degree of the $T_{ij}$ with respect to the canonical grading on $\KK[X]$. Then the weight $u_{ij}$ is the image of the weight $w_{ij}$ under the projection $K\to M$ corresponding to the inclusion $\TT\subseteq\HH$.

\smallskip

Let us begin with basic properties of the canonical grading.

\begin{lemma}
The canonical grading on $\KK[X]$ is pointed.
\end{lemma}

\begin{proof}
The assertion follows from Theorem~\ref{pcg}. Indeed, the group $K$ is a finite extension of the lattice $M$, thus $K_{\QQ}=M_{\QQ}$ and the weight cones for the $K$- and $M$-gradings on $\KK[X]$ coincide.
\end{proof}

\begin{lemma} \label{lem1}
Assume that $n_i>1$ for some $i=0,1,2$. Then indecomposable elements of the monoid $S(X)$ on the rays of the cone $\omega(X)$ are precisely the vectors $u_{ij}$ over all $i$ with $n_i>1$ and all $j=1,\ldots,n_i$.
\end{lemma}

\begin{proof}
The weights $u_{ij}$ over all $i=0,1,2$ and $j=1,\ldots,n_i$ generate the monoid $S(X)$ and the cone $\omega(X)$. If $n_k=1$ then
$$
l_{k1}u_{k1}=l_{i1}u_{i1}+\ldots+l_{in_i}u_{in_i},
$$
and the weight $u_{k1}$ is not on a ray of the cone $\omega(X)$.

To prove the assertion for $u_{ij}$ with $n_i>1$ it suffices to check that there is a linear function $\gamma$ on $M_{\QQ}$ with $\gamma(u_{ij})>0$ and $\gamma(u_{pq})\le 0$ for any $(p,q)\ne (i,j)$.

Let us fix an index $r$, $1\le r\le n_i$, $r\ne j$, and consider the one-parameter subgroup $\gamma\colon\KK^{\times}\to\TT$ with
$$
\gamma(t)\cdot T_{ij}=t^{l_{ir}}T_{ij}, \quad \gamma(t)\cdot T_{ir}=t^{-l_{ij}}T_{ir}
$$
and $\gamma(t)\cdot T_{pq}=T_{pq}$ for all other indices $p,q$. Interpreting $\gamma$ as a linear function on $M_{\QQ}$ we obtain the claim.
\end{proof}

\begin{lemma} \label{lem2}
If $n_i>1$ then $\dim\KK[X]_{u_{ij}}=1$ for any $j=1,\ldots,n_i$.
\end{lemma}

\begin{proof}
As we have seen, the weight $u_{ij}$ is indecomposable in $S(X)$, and thus the component $\KK[X]_{u_{ij}}$ does not contain monomials in $T_{pq}$ except for the element $T_{ij}$.
\end{proof}

For any element $\phi\in\Aut(X)$ we denote by $\phi^*$ the induced automorphism of the algebra $\KK[X]$. Since the action of $\Aut(X)$ normalizes the action of $\TT$, it permutes homogeneous components of the canonical grading and induces a linear action on the lattice $M$ preserving the weight monoid $S(X)$ and the weight cone $\omega(X)$. This implies that the image of $\Aut(X)$ in $\GL(M)$ is a finite group $\Gamma$.

\begin{lemma} \label{lem3}
If $n_i>1$ then for any $\phi\in\Aut(X)$ and any $j=1,\ldots,n_i$ we have $\phi^*(T_{ij})=\lambda_{i,j,\phi}T_{pq}$ for some nonzero scalar $\lambda_{i,j,\phi}$ and some
indices $p,q$.
\end{lemma}

\begin{proof}
By Lemma~\ref{lem1}, any element of the group $\Gamma$ sends $u_{ij}$ to some $u_{pq}$. Since both $\KK[X]_{u_{ij}}$ and $\KK[X]_{u_{pq}}$ are one-dimensional, we obtain the assertion.
\end{proof}

Let us consider the $M$-grading on the polynomial ring $\KK[T_{ij},i=0,1,2,j=1,\ldots,n_i]$ given by $\deg(T_{ij})=u_{ij}$. Any automorphism $\phi^*$ of the algebra $\KK[X]$ can be lifted to an endomorphism $\psi$ of the polynomial ring preserving the ideal $(f)$ and normalizing the $\TT$-action. Since the composition of liftings of $(\phi^*)^{-1}$ and $\phi^*$ is a lifting of the identity map, any lifting $\psi$ is injective. Indeed, the composition sends the generator $T_{ij}$ to $T_{ij}$ plus some element from the ideal $(f)$. All homogeneous components of an element
from $(f)$ have strictly higher degree with respect to the $M$-grading than $T_{ij}$. This shows that the composition sends every homogeneous polynomial $h$ to $h$ plus some elements of higher degrees, and thus $\psi(h)\ne 0$ for every non-zero polynomial $h$.

Let $\nu$ be the degree of (any) monomial $T_i^{l_i}$ in $f$ with respect to the $M$-grading. Clearly, the weight $\nu$ is the image of the weight $\mu$ from the previous section under the projection $K\to M$.

\begin{lemma}
The weight $\nu$ is $\Gamma$-fixed.
\end{lemma}

\begin{proof}
Let $\phi\in\Aut(X)$ and $\psi$ be a lifting of $\phi^*$ to $\KK[T_{ij}]$. We have $\psi(f)=fh$ for some homogeneous $h$, and thus the weight $\nu$ goes to $\nu+\beta$ for some $\beta\in S(X)$. Since the grading is pointed and every element of $\Gamma$ has a finite order, we conclude that $\beta=0$.
\end{proof}

\begin{proof}[Proof of Theorem~\ref{main}]
It suffices to show that for any $\phi\in\Aut(X)$ and any $T_{ij}\in\KK[X]$ we have $\phi^*(T_{ij})=\lambda_{i,j,\phi}T_{pq}$ for some nonzero scalar $\lambda_{i,j,\phi}$ and some
indices $p,q$.

If $n_i>1$ then the claim follows from Lemma~\ref{lem3}. So from now on we assume that $n_i=1$,
$j=1$, and $l_{i1}\ge 2$. In particular, the weight $u_{i1}$ is proportional to $\nu$ and thus $u_{i1}$ is $\Gamma$-fixed. Moreover, if $\dim\KK[X]_{u_{i1}}=1$, then we are done.

We may assume that $u_{i1}$ is the maximal multiple of $\nu$ with the property
$\phi^*(T_{i1})\ne \lambda T_{i1}$.

\smallskip

{\it Case~1.}\ Assume that the component $\KK[X]_{u_{i1}}$ contains at least three variables. By Lemma~\ref{lem2}, this is the case if and only if $n_0=n_1=n_2=1$ and $l_{01}=l_{11}=l_{21}=d$.
Since $X$ is rigid, we have $d\ge 3$ (see e.g. \cite[Section~3]{AKZ}), and the result follows from Example~\ref{d-3}.

\smallskip

{\it Case~2.}\ Let $\dim\KK[X]_{u_{i1}}>1$ and the component $\KK[X]_{u_{i1}}$ contain only one variable $T_{i1}$. Then $\phi^*(T_{i1})=\alpha T_{i1}+g$, where $\alpha\in\KK\setminus\{0\}$ and $g$ depends on other variables. So we have
$$
\phi^*(T_{i1}^{l_{i1}})=(\alpha T_{i1}+g)^{l_{i1}}=
\alpha^{l_{i1}}T_{i1}^{l_{i1}}+l_{i1}\alpha^{l_{i1}-1}T_{i1}^{l_{i1}-1}g+\ldots
$$
Then the term $l_{i1}\alpha T_{i1}g^{l_{i1}-1}$ can not be canceled in the image of $f$, a contradiction.

\smallskip

{\it Case~3.}\ Finally assume that the component $\KK[X]_{u_{i1}}$ contains two variables. Renumbering, we may assume that they are $T_{01}$ and $T_{11}$. Then the trinomial has the form
$$
f=T_{01}^d+T_{11}^d+T_2^{l_2}.
$$
Since $X$ is rigid, we have $d\ge 3$. We may assume that $n_2>1$ or $n_2=1$ and $\phi^*(T_{21})=\lambda T_{21}$.

By Lemma~\ref{lem3}, the ideal $J:=(T_{21},\ldots,T_{2n_2})$ is $\phi^*$-invariant. It follows from Example~\ref{d-3} that any automorphism of the factor algebra $\KK[X]/J$ can only permute the variables $T_{01}, T_{11}$ and multiply them by scalars. Taking the composition with such an automorphism we may assume that
$$
\phi^*(T_{01})=T_{01}+h_0(T_{21},\ldots,T_{2n_2}), \quad \phi^*(T_{11})=T_{11}+h_1(T_{21},\ldots,T_{2n_2}).
$$
Then the image of the term $T_{01}^d$ (resp. $T_{11}^d$) contains the term $dT_{01}h_0^{d-1}$
(resp. $dT_{11}h_1^{d-1}$), which can not be canceled. This gives a contradiction, thus $h_0=h_1=0$.

\smallskip

The proof of Theorem~\ref{main} is completed.
\end{proof}

\section{The factorial case}

In this section we consider trinomial hypersurfaces $X$ such that the algebra $\KK[X]$ is a unique factorization domain. This condition can be characterized in terms of exponents of the trinomial.
Recall that $d_i=\text{gcd}(l_{i1},\ldots,l_{in_i})$, $i=0,1,2$. Under our usual assumption $n_il_{i1}>1$ for all $i$ the trinomial hypersurface $X$ is factorial if and only if any two of $d_0,d_1,d_2$ are coprime \cite[Theorem~1.1~(ii)]{HH}. In turn, Lemma~\ref{quotient} implies that this condition is equivalent to $\HH=\TT$, or, in other words, the group $K$ is a lattice.

In this situation all homogeneous components $\KK[X]_{u_{ij}}$ are one-dimensional and the proof of Theorem~\ref{main} becomes much simpler. Moreover, in the factorial case the centralizer $C$ of the torus $\TT$ in $\Aut(X)$ coincides with~$\TT$.

By~\cite[Theorem~1.1]{Ar}, a factorial trinomial hypersurface $X$ is rigid if and only if every exponent in the trinomial $f$ is greater or equal $2$. This fact together with Theorem~\ref{main} implies the following result.

\begin{theorem} \label{fact}
Let $X$ be a trinomial hypersurface in $\AA^n$ given by $f=T_0^{l_0}+T_1^{l_1}+T_2^{l_2}$.
Assume that
\begin{itemize}
\item[(i)]
$l_{ij}\ge 2$ for all $i=0,1,2$ and $j=1,\ldots n_i$;
\item[(ii)]
the numbers $d_i=\text{gcd}(l_{i1},\ldots,l_{in_i})$, $i=0,1,2$, are pairwise coprime.
\end{itemize}
Then the automorphism group $\Aut(X)$ is isomorphic to the semi-direct product ${P(f)\rightthreetimes\TT}$, where $\TT$ is the maximal torus in $\Aut(X)$ and $P(f)$ is the permutation group of the trinomial~$f$.
\end{theorem}

\begin{example}
For the hypersurface $X$ given by
$$
T_{01}^2T_{02}^2+T_{11}^3T_{12}^3+T_{21}^5=0
$$
the automorphism group is
$$
((\ZZ/2\ZZ)\times (\ZZ/2\ZZ))\rightthreetimes \TT,
$$
where the 3-dimensional torus $\TT$ acts as
$$
(t_1,t_2,t_3)\cdot(x_{01},x_{02},x_{11},x_{12},x_{21})=
(t_1x_{01},t_1^{-1}t_3^{15}x_{02},t_2x_{11},t_2^{-1}t_3^{10}x_{12},t_3^6x_{21}),
$$
and the factors $(\ZZ/2\ZZ)$ permute $T_{01}, T_{02}$ and $T_{11}, T_{12}$, respectively.

\end{example}

\section{Isomorphisms between rigid trinomial hypersurfaces}

It is natural to ask when two trinomial hypersurfaces are isomorphic.

\begin{theorem} \label{iso}
A rigid trinomial hypersurface $f=0$ is isomorphic to a trinomial hypersurface $g=0$ if and only if the trinomial $g$ is obtained from the trinomial $f$ by a permutation of variables.
\end{theorem}

\begin{proof}
Let us denote two isomorphic trinomial hypersurfaces by $X$ and $Y$. The algebras $\KK[X]$ and $\KK[Y]$ contain the generating systems $\{T_{ij}\}$ and $\{T'_{pq}\}$ corresponding to variables in the trinomials, respectively. By dimension reason, the number of generators in both systems is the same. Let us identify the algebras $\KK[X]$ and $\KK[Y]$ and consider $\{T'_{pq}\}$ as another generating system in $\KK[X]$. Since the maximal torus in $\Aut(X)$ is unique, both $T_{ij}$ and $T'_{pg}$ are homogeneous with respect to the canonical grading on $\KK[X]$.

\smallskip

We assume first that the dimension of the weight cone $\omega(X)$ is at least 2.  Indecomposable elements $u$ of the weight monoid $S(X)$ on rays of the weight cone $\omega(X)$ are said to be \emph{extremal weights}. The corresponding homogeneous components $\KK[X]_u$ are called \emph{extremal components}. By Lemma~\ref{lem2}, extremal components are one-dimensional.

\smallskip

{\it Case 1.}\ Extremal components generate the algebra $\KK[X]$. Then Lemmas~\ref{lem1} and~\ref{lem2} imply that for any pair $(i,j)$ there exists a pair $(p,q)$ such that $T_{ij}=\lambda_{ij}T'_{pq}$, $\lambda_{ij}\in\KK\setminus\{0\}$. Generators of the ideals of relations between $T_{ij}$ and between $T'_{pq}$ are proportional, and we obtain the claim.

\smallskip

In all remaining cases we are going to show that the trinomial is uniquely determined by the graded algebra $\KK[X]$ and thus it does not depend on the set of generators.

Let $J$ be the ideal in $\KK[X]$ generated by all extremal components. By Lemma~\ref{lem1}, the ideal $J$ is generated by the elements $T_{ij}$ with $n_i>1$. So the condition $\KK[X]/J\cong\KK$ corresponds to Case~1 considered above. In general, the algebra~$\KK[X]/J$ is generated by $T_{i1}$ with $n_i=1$. If the cone $\omega(X)$ has dimension at least~2, there is an index $i$ with~$n_i>1$. Thus $\KK[X]/J$ is generated by at most two homogeneous elements.

\smallskip

{\it Case 2.}\  Let the algebra $\KK[X]/J$ be generated by one homogeneous element. Let us fix a homogeneous element $T_{01}$ in $\KK[X]$ representing a generator of $\KK[X]/J$. Having the algebra $\KK[X]/J$, we know the degree $u_{01}$ of the element $T_{01}$. Moreover, the smallest number $l_{01}$ such that $T_{01}^{l_{01}}$ is in $J$ is the exponent for $T_{01}$ in the trinomial. For the degree $\nu$ of the trinomial we have $\nu=l_{01}u_{01}$. It remains to find exponents in the other two terms.

Consider a linear combination $\sum_{i,j}c_{ij}u_{ij}$ with integer coefficients in the lattice $M$. Since $M$ is the factor group of the group $K$ by its torsion part, we have  $\sum_{i,j}c_{ij}u_{ij}=0$ if and only if there exists a positive integer $m$
such that  $m\sum_{i,j}c_{ij}w_{ij}=0$ in the group $K$ or, equivalently,
$$
m\sum_{i,j}c_{ij}e_{ij}=a_1(l_{01}e_{01}-\sum_{j=1}^{n_1}l_{1j}e_{1j})+a_2(l_{01}e_{01}-\sum_{j=1}^{n_2}l_{2j}e_{2j})
$$
in the lattice $\ZZ^n$ for some integers $a_1$ and $a_2$. This shows that there are exactly two ways to express the weight $\nu=l_{01}u_{01}$ as a linear combination of elements of a proper subset of the set of extremal weights, and the coefficients of these combinations are precisely the exponents of the monomials $T_1^{l_1}$ and $T_2^{l_2}$ respectively.

\smallskip

{\it Case 3.}\ Assume that a minimal set of homogeneous generators of the algebra $\KK[X]/J$ consists of two elements. We fix representatives $T_{01}$ and $T_{11}$ in $\KK[X]$. Knowing dimensions of homogeneous components of the algebra
$\KK[X]/J$, we find degrees of the generators $T_{01}, T_{11}$ and the degree $\nu$ of the relation between them. In turn, these data define exponents $l_{01}$ and $l_{11}$. Exponents in the third term $T_2^{l_2}$ can be found as coefficients in the expression of $\nu$ as a linear combination of extremal weights.

\smallskip

{\it Case 4.}\  Finally, let us assume that the weight cone $\omega(X)$ is one-dimensional. In this case $X$ is the trinomial surface
$$
T_{01}^{l_{01}}+T_{11}^{l_{11}}+T_{21}^{l_{21}}=0
$$
in the affine 3-space and the algebra $\KK[X]$ is $\ZZ_{\ge 0}$-graded. Since we assume that the trinomials do not contain linear terms, the degree $\nu$ of the trinomial is higher than any of the weights $u_{i1}$. Thus the dimensions of the homogeneous components of the algebra $\KK[X]$ determine uniquely the weights of the generators and of the trinomial. In turn, this allows to reconstruct exponents of the trinomial.

The proof of Theorem~\ref{iso} is completed.
\end{proof}

It is an interesting problem to find all possible isomorphisms between non-rigid trinomial hypersurfaces.

\section*{Acknowledgements}
The authors are grateful to J\"urgen Hausen, Alexander Perepechko, and Mikhail Zaidenberg for useful discussions and comments. Special thanks are due to the referee for careful reading
of the manuscript and many helpful suggestions.


\end{document}